\documentclass[a4paper]{amsart}
\usepackage{graphicx}
\usepackage{amsmath}
\usepackage{amssymb}
\usepackage{oldgerm}
\usepackage[all]{xy}

\newcommand{\Hom}{\operatorname{Hom}\nolimits}

\renewcommand{\mod}{\operatorname{mod}\nolimits}
\newcommand{\grmod}{\operatorname{grmod}\nolimits}

\newcommand{\Ext}{\operatorname{Ext}\nolimits}

\newcommand{\HH}{\operatorname{HH}\nolimits}

\newcommand{\La}{\operatorname{\Lambda}\nolimits}
\newcommand{\Ga}{\operatorname{\Gamma}\nolimits}

\newcommand{\op}{\operatorname{op}\nolimits}
\newcommand{\V}{\operatorname{V}\nolimits}

\newcommand{\T}{\operatorname{T}\nolimits}

\newcommand{\e}{\operatorname{e}\nolimits}

\newcommand{\Lae}{\operatorname{\Lambda^{\e}}\nolimits}

\newtheorem{theorem}{Theorem}[section]

\newtheorem{lemma}[theorem]{Lemma}
\newtheorem{proposition}[theorem]{Proposition}
\theoremstyle{definition}

\theoremstyle{definition}

\theoremstyle{definition}
\newtheorem*{examples}{Examples}

\theoremstyle{remark}

\theoremstyle{remark}

\theoremstyle{definition}

\begin{document}
\title{$\Ext$-symmetry over quantum complete intersections}
\author{Petter Andreas Bergh}
\address{Petter Andreas Bergh \\ Institutt for matematiske fag \\
NTNU \\ N-7491 Trondheim \\ Norway} \email{bergh@math.ntnu.no}

\thanks{The author was supported by NFR Storforsk grant no.\
167130}



\maketitle

\begin{abstract}
We show that symmetry in the vanishing of cohomology holds for
graded modules over quantum complete intersections. Moreover,
symmetry holds for all modules if the algebra is symmetric.
\end{abstract}

\section{Introduction}\label{sec1}

For which algebras do symmetry in the vanishing of cohomology hold?
That is, given an algebra $\La$, does the implication
$$\Ext_{\La}^i(M,N)=0 \text{ for } i \gg 0 \hspace{2mm}
\Rightarrow \hspace{2mm} \Ext_{\La}^i(N,M)=0 \text{ for } i \gg
0$$ hold for finitely generated modules $M$ and $N$? As shown in
\cite{AvramovBuchweitz}, this implication holds for finitely
generated modules over commutative local complete intersections.
The proof involves the theory of certain support varieties
attached to each pair of finitely generated modules over such a
ring $A$. Namely, denote by $c$ the codimension of $A$ and by $K$
the algebraic closure of its residue field. A cone $\V (M,N)$ in
$K^c$ is associated to every pair $(M,N)$ of finitely generated
$A$-modules, with the following properties:
\begin{eqnarray*}
\V (M,N)= \{ 0 \} & \Leftrightarrow & \Ext_A^i (M,N)=0 \text{
for } i \gg 0, \\
\V (M,N) & = & \V (M,M) \cap \V (N,N).
\end{eqnarray*}
The symmetry in the vanishing of cohomology follows immediately from
these properties. As shown in \cite[Corollary 4.8]{Mori}, group
algebras of finite groups provide another class of examples where
$\Ext$-symmetry holds.

We show in this paper that $\Ext$-symmetry holds for all
\emph{graded} modules over quantum complete intersections, provided
all the defining commutators are roots of unity. We also show that,
if such an algebra is symmetric, that is, if it is isomorphic as a
bimodule to its own dual, then symmetry holds for all modules.

\section{Quantum complete intersections}\label{sec2}

Quantum complete intersections are noncommutative analogues of
truncated polynomial rings. Included in this class of algebras are
exterior algebras and finite dimensional complete intersections of
the form $k[X_1, \dots, X_c]/(X_1^{a_1}, \dots, X_c^{a_c})$. Fix a
field $k$, let $c \ge 1$ be an integer, and let ${\bf{q}} =
(q_{ij})$ be a $c \times c$ commutation matrix with entries in $k$.
That is, the diagonal entries $q_{ii}$ are all $1$, and
$q_{ij}q_{ji}=1$ for all $i,j$. Furthermore, let ${\bf{a}}_c = (a_1,
\dots, a_c)$ be an ordered sequence of $c$ integers with $a_i \ge
2$. The \emph{quantum complete intersection}
$A_{\bf{q}}^{{\bf{a}}_c}$ determined by these data is the algebra
$$A_{\bf{q}}^{{\bf{a}}_c} \stackrel{\text{def}}{=} k \langle X_1,
\dots, X_c \rangle / (X_i^{a_i}, X_iX_j-q_{ij}X_jX_i).$$ This is a
finite dimensional selfinjective algebra of dimension
$\prod_{i=1}^c a_i$. The image of $X_i$ in this quotient will be
denoted by $x_i$. We shall consider $A_{\bf{q}}^{{\bf{a}}_c}$ as a
$\mathbb{Z}^c$-graded algebra, in which the degree of the
generator $x_i$ is the $i$th unit vector $(0, \dots, 1, \dots 0)$.
The category of finitely generated left
$A_{\bf{q}}^{{\bf{a}}_c}$-modules (respectively, graded modules)
is denoted by $\mod A_{\bf{q}}^{{\bf{a}}_c}$ (respectively,
$\grmod A_{\bf{q}}^{{\bf{a}}_c}$); all modules are assumed to be
finitely generated.

A quantum complete intersection is built from truncated polynomial
rings using certain tensor products. We recall here the basics,
details can be found in \cite{BerghOppermann}. Let $A$ and $B$ be
abelian groups, and let $\La$ and $\Ga$ be an $A$-graded and a
$B$-graded $k$-algebra, respectively. Furthermore, let $t \colon A
\otimes_{\mathbb{Z}} B \to k \setminus \{ 0 \}$ be a homomorphism
of groups, where $k \setminus \{ 0 \}$ is multiplicative. If $a$
and $b$ are elements of $A$ and $B$, respectively, then we write
$t(a|b)$ instead of $t(a \otimes b)$. Moreover, given homogeneous
elements $\lambda \in \La$ and $\gamma \in \Ga$ of degrees $d$ and
$d'$, respectively, we write $t( \lambda | \gamma )$ instead of
$t(d|d')$. With these data, we can now define a new algebra $\La
\otimes_k^t \Ga$, the \emph{twisted tensor product} of $\La$ and
$\Ga$ with respect to the homomorphism $t$. The underlying
$k$-vector space of $\La \otimes_k^t \Ga$ is $\La \otimes_k \Ga$,
and multiplication is given by
$$( \lambda_1 \otimes \gamma_1 ) \cdot ( \lambda_2 \otimes \gamma_2
) \stackrel{\text{def}}{=} t( \lambda_2 | \gamma_1 ) \lambda_1
\lambda_2 \otimes \gamma_1 \gamma_2$$ for homogeneous elements
$\lambda_1, \lambda_2 \in \La$ and $\gamma_1, \gamma_2 \in \Ga$.
This algebra is $A \oplus B$-graded; if $a \in A$ and $b \in B$,
then $( \La \otimes_k^t \Ga )_{(a,b)} = \La_a \otimes_k \Ga_b$.
Given a graded $\La$-module $M$ and a graded $\Ga$-module $N$,
their tensor product $M \otimes_k N$ becomes a graded $\La
\otimes_k^t \Ga$-module by defining
$$( \lambda \otimes \gamma ) \cdot ( m \otimes n )
\stackrel{\text{def}}{=} t(m| \gamma ) \lambda m \otimes \gamma n$$
for homogeneous elements $\lambda \in \La, \gamma \in \Ga, m \in M,
n \in N$. This module is denoted $M \otimes_k^t N$. As above, this
module is $A \oplus B$-graded; if $a \in A$ and $b \in B$, then $(M
\otimes_k^t N)_{(a,b)} = M_a \otimes_k N_b$.

The purpose of this paper is to study symmetry in the vanishing of
cohomology over quantum complete intersections. We therefore end
this section with the following two results, the first of which
shows that quantum complete intersections are made up of twisted
tensor products.

\begin{lemma}\label{twisted}
Let $A_{\bf{q}}^{{\bf{a}}_c}$ be a quantum complete intersection
with $c \ge 2$, let $I$ be a proper nonempty subset of $\{ 1, \dots,
c \}$ of order $c_1$, and let $c_2 = c-c_1$. Furthermore, let
$A_{{\bf{q}}_1}^{{\bf{a}}_{c_1}}$ and
$A_{{\bf{q}}_2}^{{\bf{a}}_{c_2}}$ be the quantum complete
intersections generated by $\{ x_i \}_{i \in I}$ and $\{ x_i \}_{i
\in \{ 1, \dots, c \} \setminus I}$, respectively. Then there is an
isomorphism
$$A_{\bf{q}}^{{\bf{a}}_c} \simeq A_{{\bf{q}}_1}^{{\bf{a}}_{c_1}}
\otimes_k^t A_{{\bf{q}}_2}^{{\bf{a}}_{c_2}}$$ for some
homomorphism $\mathbb{Z}^{c_1} \otimes_{\mathbb{Z}}
\mathbb{Z}^{c_2} \to k \setminus \{ 0 \}$.
\end{lemma}

\begin{proof}
By re-indexing the generators, we may assume that $I= \{ 1, \dots,
c_1 \}$. Thus $A_{{\bf{q}}_1}^{{\bf{a}}_{c_1}}$ is the subalgebra of
$A_{\bf{q}}^{{\bf{a}}_c}$ generated by $x_1, \dots, x_{c_1}$,
whereas $A_{{\bf{q}}_2}^{{\bf{a}}_{c_2}}$ is the subalgebra
generated by the $c_2$ elements $x_{c_1+1}, \dots, x_c$.

Consider the map $\mathbb{Z}^{c_1} \times \mathbb{Z}^{c_2}
\xrightarrow{t'} k \setminus \{ 0 \}$ defined by
$$\left ( (d_1, \dots, d_{c_1}),(d_{c_1+1}, \dots, d_c) \right )
\mapsto \prod_{j=c_1+1}^c \prod_{i=1}^{j-1}q_{ji}^{d_id_j}.$$ Given
sequences ${\bf{d}}_1, {\bf{d}}_1' \in \mathbb{Z}^{c_1}$ and
${\bf{d}}_2, {\bf{d}}_2' \in \mathbb{Z}^{c_2}$, the equalities
\begin{eqnarray*}
t'( {\bf{d}}_1 + {\bf{d}}_1', {\bf{d}}_2 ) & = & t'( {\bf{d}}_1,
{\bf{d}}_2 ) t'( {\bf{d}}_1', {\bf{d}}_2 ) \\
t'( {\bf{d}}_1, {\bf{d}}_2 + {\bf{d}}_2' ) & = & t'( {\bf{d}}_1,
{\bf{d}}_2 ) t'( {\bf{d}}_1, {\bf{d}}_2' )
\end{eqnarray*}
hold, hence $t'$ induces a homomorphism $\mathbb{Z}^{c_1}
\otimes_{\mathbb{Z}} \mathbb{Z}^{c_2} \xrightarrow{t} k \setminus \{
0 \}$ of abelian groups. Now let $i$ and $j$ be elements of $I$, and
consider the elements $x_i \otimes 1$ and $x_j \otimes 1$ in the
twisted tensor product $A_{{\bf{q}}_1}^{{\bf{a}}_{c_1}} \otimes_k^t
A_{{\bf{q}}_2}^{{\bf{a}}_{c_2}}$. The degrees of $x_i \in
A_{{\bf{q}}_1}^{{\bf{a}}_{c_1}}$ and $x_j \in
A_{{\bf{q}}_1}^{{\bf{a}}_{c_1}}$ are the unit vectors $e_i$ and
$e_j$ in $\mathbb{Z}^{c_1}$, respectively, whereas the degree of $1
\in A_{{\bf{q}}_2}^{{\bf{a}}_{c_2}}$ is the zero vector in
$\mathbb{Z}^{c_2}$. Therefore
\begin{eqnarray*}
(x_i \otimes 1)(x_j \otimes 1) & = & t(e_j|0) x_ix_j \otimes 1 \\
& = & x_ix_j \otimes 1 \\
& = & q_{ij} x_jx_i \otimes 1 \\
& = & q_{ij} t(e_i|0) x_jx_i \otimes 1 \\
& = & q_{ij}(x_j \otimes 1 )(x_i \otimes 1),
\end{eqnarray*}
and similarly $(1 \otimes x_i)(1 \otimes x_j) = q_{ij}(1 \otimes
x_j)(1 \otimes x_i)$ whenever $i$ and $j$ are both in $\{ c_1+1,
\dots, c \}$. If $i \in I$ and $j \in \{ c_1+1, \dots, c \}$, then
the degree of $x_j$ in $A_{{\bf{q}}_2}^{{\bf{a}}_{c_2}}$ is the unit
vector $e_j$ in $\mathbb{Z}^{c_2}$. In this case we obtain the
equalities
$$(x_i \otimes 1)(1 \otimes x_j) = t(0|0)x_i \otimes x_j = x_i
\otimes x_j$$ and
$$(1 \otimes x_j)(x_i \otimes 1) = t(e_i|e_j)x_i \otimes x_j = q_{ji}x_i
\otimes x_j,$$ and consequently $A_{\bf{q}}^{{\bf{a}}_c}$ is
isomorphic to $A_{{\bf{q}}_1}^{{\bf{a}}_{c_1}} \otimes_k^t
A_{{\bf{q}}_2}^{{\bf{a}}_{c_2}}$.
\end{proof}

The final result of this section shows that the cohomology of a
twisted tensor product of graded modules is the tensor product of
the respective cohomologies.

\begin{theorem}\cite[Theorem
3.7]{BerghOppermann}\label{cohomology} Let $A$ and $B$ be abelian
groups, and let $\La$ and $\Ga$ be an $A$-graded and a $B$-graded
$k$-algebra, respectively. Furthermore, let $t \colon A
\otimes_{\mathbb{Z}} B \to k \setminus \{ 0 \}$ be a homomorphism,
and let $M_1,M_2 \in \grmod \La$ and $N_1,N_2 \in \grmod \Ga$ be
graded modules. Then there is an isomorphism
$$\Ext_{\La \otimes_k^t \Ga}^*(M_1 \otimes_k^t N_1,M_2 \otimes_k^t
N_2) \simeq \Ext_{\La}^*(M_1,M_2) \otimes_k
\Ext_{\Ga}^*(N_1,N_2)$$ of graded vector spaces.
\end{theorem}

\section{$\Ext$-symmetry}

In this section, we prove that symmetry holds in the vanishing of
cohomology for graded modules over a quantum complete
intersection, provided the commutators are all roots of unity. The
idea of the proof is to use twisted tensor products to pass to a
bigger quantum complete intersection where $\Ext$-symmetry holds,
namely a symmetric one. In order to do this, we must determine
precisely when a quantum complete intersection is symmetric.
Throughout this section, we fix a field $k$.

Recall that a finite dimensional $k$-algebra $\La$ is
\emph{Frobenius} if ${_{\La}\La}$ and $D(\La_{\La})$ are isomorphic
as left $\La$-modules, where $D$ denotes the usual $k$-dual
$\Hom_k(-,k)$. If $\La$ and $D( \La )$ are isomorphic as bimodules,
then $\La$ is \emph{symmetric}. Now suppose $\La$ is Frobenius, and
let $\phi \colon _{\La}\La \to D(\La_{\La})$ be an isomorphism. Let
$y \in \La$ be any element, and consider the linear functional
$\phi(1) \cdot y \in D(\La)$, i.e.\ the $k$-linear map $\La \to k$
defined by $\lambda \mapsto \phi(1)(y \lambda)$. Since $\phi$ is
surjective, there is an element $x \in \La$ having the property that
$\phi(x) = \phi(1) \cdot y$, giving $x \cdot \phi(1) = \phi(1) \cdot
y$ since $\phi$ is a map of left $\La$-modules. It is not difficult
to show that the map $y \mapsto x$ defines a $k$-algebra
automorphism on $\La$, and its inverse $\nu$ is called the
\emph{Nakayama automorphism} of $\La$ (with respect to $\phi$). Thus
$\nu$ is defined by $\phi(1)(\lambda x) = \phi(1)( \nu(x) \lambda)$
for all $\lambda \in \La$. This automorphism is unique up to an
inner automorphism; if $\phi' \colon _{\La}\La \to D(\La_{\La})$ is
another isomorphism of left modules yielding a Nakayama automorphism
$\nu'$, then there exists an invertible element $z \in \La$ such
that $\nu = z \nu' z^{-1}$. Note that $\phi$ is an isomorphism
between the bimodules $_1\La_{\nu^{-1}}$ and $D(\La)$. Moreover,
note that $\La$ is symmetric if and only if the Nakayama
automorphism is the identity.

As $D(\La_{\La})$ is an injective left $\La$-module, we see that a
Frobenuis algebra is always left selfinjective, but in fact the
definition is left-right symmetric. For if $\phi \colon _{\La}\La
\to D(\La_{\La})$ is an isomorphism of left $\La$-modules, we can
dualize and get an isomorphism $D( \phi ) \colon D^2( \La_{\La} )
\to D( _{\La}\La )$ of right modules. Composing with the natural
isomorphism $\La_{\La} \simeq D^2( \La_{\La} )$, we obtain an
isomorphism $\La_{\La} \to D( _{\La}\La)$ of right $\La$-modules.

A finite dimensional  local algebra is Frobenius if and only if it
is selfinjective. In particular, a quantum complete intersection
$A_{\bf{q}}^{{\bf{a}}_c}$ is Frobenius, and the following result
shows that there is a particularly nice Nakayama automorphism.

\begin{lemma}\label{nakayama}
A quantum complete intersection $A_{\bf{q}}^{{\bf{a}}_c}$ is
Frobenius, with a Nakayama automorphism $A_{\bf{q}}^{{\bf{a}}_c}
\xrightarrow{\nu} A_{\bf{q}}^{{\bf{a}}_c}$ given by
$$x_w \mapsto \left ( \prod_{i=1}^c q_{iw}^{a_i-1} \right ) x_w$$
for $1 \le w \le c$.
\end{lemma}

\begin{proof}
Consider the map $A_{\bf{q}}^{{\bf{a}}_c} \xrightarrow{\phi} D (
A_{\bf{q}}^{{\bf{a}}_c} )$ defined by
$$\phi (1) \colon \sum_{i_1, \dots, i_c} \alpha_{i_1, \dots, i_c} x_c^{i_c} \cdots x_1^{i_1} \mapsto \alpha_{a_1-1, \dots,
a_c-1}.$$ That is, the element $\phi (1)$ maps an element $y \in
A_{\bf{q}}^{{\bf{a}}_c}$ to the coefficient of the socle element
$x_c^{a_c-1} \cdots x_1^{a_1-1}$ in $y$. This is an isomorphism of
left $A_{\bf{q}}^{{\bf{a}}_c}$-modules. By definition, a Nakayama
automorphism $A_{\bf{q}}^{{\bf{a}}_c} \xrightarrow{\nu}
A_{\bf{q}}^{{\bf{a}}_c}$ has the property that $y \cdot \phi (1) =
\phi (1) \cdot \nu (y)$ for all $y \in A_{\bf{q}}^{{\bf{a}}_c}$. The
given map satisfies this property.
\end{proof}

Thus quantum complete intersections are not symmetric in general.
However, the following result shows that for every such algebra,
there exists a symmetric quantum complete intersection
``extending" the given one.

\begin{proposition}\label{symmetric}
Given any quantum complete intersection $A_{\bf{q}}^{{\bf{a}}_c}$,
there exists a symmetric quantum complete intersection
$A_{\bf{q}'}^{{\bf{a}}_{2c}}$ with the following properties:
\begin{itemize}
\item[(i)] The subalgebra of $A_{\bf{q}'}^{{\bf{a}}_{2c}}$ generated
by $x_1, \dots, x_c$ is isomorphic to $A_{\bf{q}}^{{\bf{a}}_c}$.
\item[(ii)] The commutators of $A_{\bf{q}'}^{{\bf{a}}_{2c}}$ are the
commutators of $A_{\bf{q}}^{{\bf{a}}_c}$.
\end{itemize}
\end{proposition}

\begin{proof}
Suppose $A_{\bf{q}}^{{\bf{a}}_c}$ is given by the sequence
${\bf{a}}_c =(a_1, \dots, a_c)$ and the commutation matrix
$${\bf{q}} = \left (
\begin{array}{ccc}
q_{11} & \cdots & q_{1c} \\
\vdots & \ddots & \vdots \\
q_{c1} & \cdots & q_{cc}
\end{array}
\right )$$ in which $q_{ij}q_{ji}=1$ and $q_{ii}=1$ for all $i,j$.
Define a sequence ${\bf{a}}_{2c}$ and a $2c \times 2c$ commutation
matrix ${\bf{q}'}$ by
\begin{eqnarray*}
{\bf{a}}_{2c} & \stackrel{\text{def}}{=} & (a_1, \dots, a_c,a_1,
\dots, a_c) \\
{\bf{q}'} & \stackrel{\text{def}}{=} & \left (
\begin{array}{cc}
{\bf{q}} & {\bf{q}}^{\T} \\
{\bf{q}}^{\T} & {\bf{q}}
\end{array}
\right )
\end{eqnarray*}
where $Q^{\T}$ denotes the transpose of a matrix. Then
$A_{\bf{q}}^{{\bf{a}}_c}$ is isomorphic to the subalgebra of the
quantum complete intersection $A_{\bf{q}'}^{{\bf{a}}_{2c}}$
generated by $x_1, \dots, x_c$. The commutators $q_{uv}'$ in
$A_{\bf{q}'}^{{\bf{a}}_{2c}}$ satisfy
$$q_{uv}' = \left \{
\begin{array}{ll}
q_{uv} & \text{ if } u \le c, v \le c \\
q_{vi} & \text{ if } u=c+i, v \le c \\
q_{ju} & \text{ if } u \le c, v=c+j \\
q_{ij} & \text{ if } u=c+i, v=c+j,
\end{array}
\right.$$ and so if $i$ and $w$ are integers with $1 \le i \le c$
and $1 \le w \le 2c$, we see that $q'_{iw}q'_{(c+i)w}=1$. This gives
$$\prod_{i=1}^{2c}(q_{iw}')^{a_i-1} = \left (
\prod_{i=1}^{c}(q_{iw}')^{a_i-1} \right ) \left (
\prod_{i=1}^{c}(q_{(c+i)w}')^{a_i-1} \right ) =1$$ for $1 \le w \le
2c$, and therefore $A_{\bf{q}'}^{{\bf{a}}_{2c}}$ is symmetric by
Lemma \ref{nakayama}.
\end{proof}

As mentioned, the crucial step when proving $\Ext$-symmetry for
graded modules over a quantum complete intersection is the passing
to a bigger symmetric algebra, as in the previous result. Namely,
the next result shows that $\Ext$-symmetry holds for \emph{all}
modules over a symmetric quantum complete intersection whose
commutators are all roots of unity. Recall first the following;
details can be found in \cite{SnashallSolberg} and \cite{Solberg}.
Let $\La$ a finite dimensional $k$-algebra, and denote the
enveloping algebra $\La \otimes_k \La^{\op}$ of $\La$ by $\Lae$. For
$n \ge 0$, the $n$th \emph{Hochschild cohomology} group of $\La$,
denoted $\HH^n ( \La )$, is the vector space $\Ext_{\Lae}^n( \La,
\La )$. The graded vector space $\HH^* ( \La ) = \Ext_{\Lae}^*( \La,
\La )$ is a graded-commutative ring with Yoneda product, and for
every $M \in \mod \La$ the tensor product $- \otimes_{\La} M$
induces a homomorphism
$$\HH^* ( \La ) \xrightarrow{\varphi_M} \Ext_{\La}^*(M,M)$$
of graded $k$-algebras. If $N \in \mod \La$ is another module and
$\eta \in \HH^* ( \La )$ and $\theta \in \Ext_{\La}^*(M,N)$ are
homogeneous elements, then the relation $\varphi_N( \eta ) \circ
\theta = (-1)^{|\eta||\theta|} \theta \circ \varphi_M( \eta )$
holds, where ``$\circ$" denotes the Yoneda product. In the
terminology used in \cite{Bergh} and \cite{BIKO}, the Hochschild
cohomology ring $\HH^* ( \La )$ acts centrally on the bounded
derived category $D^b( \La )$ of $\mod \La$.

\begin{theorem}\label{symmetrysymmetric}
Let $A_{\bf{q}}^{{\bf{a}}_c}$ be a symmetric quantum complete
intersection whose commutators $q_{ij}$ are all roots of unity.
Then for all modules $M,N \in \mod A_{\bf{q}}^{{\bf{a}}_c}$, the
following are equivalent:
\begin{enumerate}
\item[(i)] $\Ext_{A_{\bf{q}}^{{\bf{a}}_c}}^i(M,N)=0$ for $i \gg
0$. \item[(ii)] $\Ext_{A_{\bf{q}}^{{\bf{a}}_c}}^i(M,N)=0$ for $i
>0$. \item[(iii)] $\Ext_{A_{\bf{q}}^{{\bf{a}}_c}}^i(N,M)=0$ for $i \gg 0$. \item[(iv)]
$\Ext_{A_{\bf{q}}^{{\bf{a}}_c}}^i(N,M)=0$ for $i > 0$.
\end{enumerate}
\end{theorem}

\begin{proof}
By \cite[Theorem 5.5]{BerghOppermann}, the Hochschild cohomology
ring $\HH^* ( A_{\bf{q}}^{{\bf{a}}_c} )$ is Noetherian, and
$\Ext_{A_{\bf{q}}^{{\bf{a}}_c}}^i(k,k)$ is a finitely generated
$\HH^* ( A_{\bf{q}}^{{\bf{a}}_c} )$-module. The result now follows
from \cite[Theorem 4.2]{Bergh}.
\end{proof}

\begin{examples}
(i) Let $A$ be the exterior algebra on a $c$-dimensional $k$-vector
space, i.e.\
$$A = k \langle X_1, \dots, X_c \rangle /(X_i^2, \{ X_iX_j+X_jX_i
\}_{i \neq j} ).$$ From Lemma \ref{nakayama}, we see that when
applying the Nakayama automorphism to a generator $x_i$, then the
result is the element $(-1)^{c-1} x_i$. Therefore $A$ is symmetric
precisely when $c$ is an odd number. Consequently, symmetry in the
vanishing of cohomology holds for modules over exterior algebras on
odd-dimensional vector spaces (cf.\ \cite[Corollary 4.9]{Mori}).

(ii) Fix integers $c$ and $a$, both at least two, and let $q$ be an
element in $k$ with the property that $q^{a-1}=1$. Furthermore, let
${\bf{a}}_c$ be the $c$-tuple $(a, \dots, a)$, let ${\bf{q}}$ be the
commutation matrix
$$\left (
\begin{array}{ccccc}
1 & q & \cdots & \cdots & q \\
q^{-1} & 1 & q & \cdots & q \\
\vdots & \ddots & \ddots & \ddots & \vdots \\
q^{-1} & \cdots & q^{-1} & 1 & q \\
q^{-1} & \cdots & \cdots & q^{-1} & 1
\end{array}
\right )$$ and consider the quantum complete intersection
$A_{\bf{q}}^{{\bf{a}}_c}$. Explicitly, this is the algebra
$$k \langle X_1, \dots, X_c \rangle /(X_i^a, \{ X_iX_j-qX_jX_i \}_{i<j} ),$$
and from Lemma \ref{nakayama} we see that it is symmetric. Thus
$\Ext$-symmetry holds for all modules over
$A_{\bf{q}}^{{\bf{a}}_c}$.
\end{examples}

Using Theorem \ref{symmetrysymmetric}, we now show that
$\Ext$-symmetry holds for all \emph{graded} modules over an
arbitrary quantum complete intersection whose commutators are all
roots of unity.

\begin{theorem}\label{symmetrygraded}
Let $A_{\bf{q}}^{{\bf{a}}_c}$ be a quantum complete intersection
whose commutators $q_{ij}$ are all roots of unity. Then for all
modules $M,N \in \grmod A_{\bf{q}}^{{\bf{a}}_c}$, the following are
equivalent:
\begin{enumerate}
\item[(i)] $\Ext_{A_{\bf{q}}^{{\bf{a}}_c}}^i(M,N)=0$ for $i \gg
0$. \item[(ii)] $\Ext_{A_{\bf{q}}^{{\bf{a}}_c}}^i(M,N)=0$ for $i
>0$. \item[(iii)] $\Ext_{A_{\bf{q}}^{{\bf{a}}_c}}^i(N,M)=0$ for $i \gg 0$. \item[(iv)]
$\Ext_{A_{\bf{q}}^{{\bf{a}}_c}}^i(N,M)=0$ for $i > 0$.
\end{enumerate}
\end{theorem}

\begin{proof}
Let $A_{\bf{q}'}^{{\bf{a}}_{2c}}$ be a symmetric quantum complete
intersection with the properties given in Proposition
\ref{symmetric}. That is, the subalgebra of
$A_{\bf{q}'}^{{\bf{a}}_{2c}}$ generated by $x_1, \dots, x_c$ is
isomorphic to $A_{\bf{q}}^{{\bf{a}}_c}$, and the commutators of
$A_{\bf{q}'}^{{\bf{a}}_{2c}}$ are the commutators of
$A_{\bf{q}}^{{\bf{a}}_c}$. By Lemma \ref{twisted}, there exists a
homomorphism $\mathbb{Z}^c \otimes_{\mathbb{Z}} \mathbb{Z}^c
\xrightarrow{t} k \setminus \{ 0 \}$ and a quantum complete
intersection $A_{{\bf{p}}}^{{\bf{b}}_c}$, such that
$A_{\bf{q}'}^{{\bf{a}}_{2c}}$ is isomorphic to the twisted tensor
product $A_{\bf{q}}^{{\bf{a}}_c} \otimes_k^t
A_{{\bf{p}}}^{{\bf{b}}_c}$. Now for graded
$A_{\bf{q}}^{{\bf{a}}_c}$-modules $M$ and $N$, Theorem
\ref{cohomology} provides an isomorphism
$$\Ext_{A_{\bf{q}'}^{{\bf{a}}_{2c}}}^*( M \otimes_k^t
A_{{\bf{p}}}^{{\bf{b}}_c}, N \otimes_k^t A_{{\bf{p}}}^{{\bf{b}}_c} )
\simeq \Ext_{A_{\bf{q}}^{{\bf{a}}_c}}^*(M,N) \otimes_k
\Ext_{A_{{\bf{p}}}^{{\bf{b}}_c}}^*( A_{{\bf{p}}}^{{\bf{b}}_c},
A_{{\bf{p}}}^{{\bf{b}}_c} )$$ of graded vector spaces. The result
now follows from Theorem \ref{symmetrysymmetric}, since all the
commutators of $A_{\bf{q}'}^{{\bf{a}}_{2c}}$ are roots of unity.
\end{proof}

\end{document}